\documentclass[12pt]{article}

\usepackage{mathdesign}

\usepackage{color}

\newcommand{\black}{\color{black}}

\usepackage{amssymb}
\usepackage{epstopdf}
\usepackage{amsmath}
\usepackage{amsthm}
\usepackage{amsfonts}
\usepackage{epstopdf}

\newtheorem{theorem}{Theorem}
\newtheorem{lemma}[theorem]{Lemma}

\newtheorem{corollary}[theorem]{Corollary}

\newtheorem{proposition}[theorem]{Proposition}
\theoremstyle{definition}

\theoremstyle{remark}
\newtheorem{remark}[theorem]{Remark}

\newcommand\beq{\begin {equation}}
\newcommand\eeq{\end {equation}}
\newcommand\beqs{\begin {equation*}}
\newcommand\eeqs{\end {equation*}}

\newcommand\R{{\mathbb R}}

\newcommand\E{{\mathbb E}}
\newcommand\sph{{\mathbb S}}

\newcommand\eu{{\rm e}}

\vspace{1in}

\black

\title{\bf {A Stein deficit for \\ the logarithmic Sobolev inequality}}

\author{Michel Ledoux\footnote{{Institut de Math\'ematiques de Toulouse, Universit\'e
de Toulouse -- Paul-Sabatier, F-31062 Toulouse, France \& Institut Universitaire de France} {\tt   ledoux@math.univ-toulouse.fr}}\quad\quad\quad Ivan Nourdin\footnote{{Facult\'e des Sciences, de la Technologie et de la Communication; UR en Math\'ematiques. Luxembourg University, 6, rue Richard Coudenhove-Kalergi, L-1359 Luxembourg} {\tt ivan.nourdin@uni.lu}; IN was partially supported by the Grant F1R-MTH-PUL-15CONF  (CONFLUENT) at Luxembourg University. }\quad  \quad\quad
Giovanni Peccati\footnote{{Facult\'e des Sciences, de la Technologie et de la Communication; UR en Math\'ematiques. Luxembourg University, 6, rue Richard Coudenhove-Kalergi, L-1359 Luxembourg} {\tt giovanni.peccati@gmail.com}; GP was partially supported by the grant F1R-MTH-PUL-15STAR (STARS) at Luxembourg University.}}

\begin{document}
\maketitle
\begin{abstract} {We provide explicit lower bounds for the deficit in the Gaussian
logarithmic Sobolev inequality in terms of differential operators that are naturally
associated with the so-called Stein characterization of the Gaussian distribution.
The techniques are based on a crucial use of the representation of the relative
Fisher information, along the Ornstein-Uhlenbeck semigroup, in terms of the
Minimal Mean-Square Error from information theory.}

\bigskip


\end{abstract}

\bibliographystyle{plain}

\maketitle



\vskip 10mm

\section {Introduction and main results} \label {sec.1}

The classical logarithmic Sobolev inequality for the standard Gaussian measure
$$
d\gamma(x) \, = \,  \eu^{-|x|^2/2} \frac {dx}{(2\pi)^{n/2}}
$$
on the Borel sets of $\R^n$ expresses that for any
smooth probability density $f$ with respect to $\gamma $,
\beq \label {eq.logsob}
 {\rm H} (f) \, = \, \int_{\R^n} f \log f \, d\gamma 
      \, \leq \, \frac {1}{2} \, \int_{\R^n} \frac {|\nabla f|^2}{f} \, d\gamma
        \, = \, \frac {1}{2} \, {\rm I} (f)
\eeq
where ${\rm H}(f)$ is the relative entropy of the measure $f d\gamma$ with respect to
$\gamma $ and ${\rm I}(f)$ is its Fisher information.

It is a classical result, due to E.~Carlen \cite {C91a,C91b}, that the exponential densities
\beq \label {eq.extremal}
e_b (x) \, = \,  \eu^{b\cdot x - |b|^2/2}, \quad x \in \R^n, \, \, \, b \in \R^n,
\eeq
are saturating the inequality \eqref {eq.logsob} and are the only ones.
Note that the probability density $e_b$ with respect to $\gamma$
has mean $b$ and covariance matrix the identity ${\rm Id}$.

Modulo smoothness assumptions on the underlying density $f$,
a proof of this result may be given by interpolation along the
Ornstein-Uhlenbeck semigroup (cf.~\cite {L91}). Let ${(P_t)}_{t \geq 0}$ be the
Ornstein-Uhlenbeck semigroup with integral representation
\beq \label {eq.ou}
P_t g(x) \, = \,  \int_{\R^n} g \big ( \eu^{-t} x + \sqrt {1 - \eu^{-2t}} \, y \big ) d\gamma (y),
   \quad t \geq 0, \, \, x \in \R^n ,
\eeq
for any suitable $ g : \R^n \to \R$. By expansion along this semigroup, the Bakry-\'Emery calculus
(see \cite {B94,BE85,BGL14,L91} and below) yields that
\beq \label {eq.representation}
 {\rm H} (f)  \, = \,  \frac {1}{2} \, {\rm I} (f) - 
      \int_0^\infty \! \int_{\R^n} P_t f  \, \big | {\rm Hess} (\log P_t f) \big |^2 d\gamma \, dt.
\eeq
Here and throughout this work, $|\cdot|$ denotes the Euclidean norm on vectors and
matrices (Hilbert-Schmidt norm).
Hence, if there is equality in \eqref {eq.logsob}, for almost every $t \geq 0$ and $x \in \R^n$,
$ {\rm Hess} (\log P_t f) (x) = 0 $ so that $\log f (x)$ is affine.

Following recent investigations for classical Sobolev and isoperimetric inequalities,
both for the Lebesgue and Gaussian measures
\cite {BBJ14,CFMP09,E14,FMP10, FMP08,MN15}, the question has been raised to
quantify the deficit in the logarithmic Sobolev inequality via a suitable distance to
the saturating exponential densities.
To this task, introduce, for a (smooth) probability density
$f$ with respect to $\gamma$, the deficit
\beq \label {eq.deficit}
 \delta (f) \, = \,   \frac {1}{2} \, {\rm I} (f) -  {\rm H} (f ) \, \geq \, 0
\eeq
in the logarithmic Sobolev inequality \eqref {eq.logsob} for the density $f$. We speak equivalently
of the deficit of the probability $d\mu = f d\gamma$.
Relevant lower bounds on the deficit may then be interpreted as a stability estimate
on the functional inequality with respect to the extremizers.

Stability in the logarithmic Sobolev
has therefore motivated recently a number of investigations.
However, the various conclusions so far do not appear fully satisfactory, in particular with
respect to dimension free bounds which should reasonably be expected
(as the logarithmic Sobolev inequality itself does not depend on the dimension of the underlying
state space).
Note in particular that the corresponding study of the deficit in the Gaussian 
isoperimetric inequality, after the first investigation in \cite {MN15}, finally produced (optimal)
dimension free bounds \cite {BBJ14,E14} with respect to a natural distance
to the extremal sets (half-spaces). While the logarithmic Sobolev inequality may be derived from
the Gaussian isoperimetric inequality (cf.~\cite {BGL14}), the derivation does not
seem to preserve any information on the deficit.

To briefly survey some of the recent conclusions on the deficit in the logarithmic Sobolev inequality,
note first the lower bound, under the condition $\int_{\R^n} |x|^2 d\mu \leq n$,
\beq \label {eq.bgrs}
 \delta (f) \, \geq \, \frac {1}{100 \,n} \, {\rm W}_2(\mu, \gamma)^4
\eeq
emphasized in \cite {BGRS14} after an inequality of \cite {BL06}, where ${\rm W}_2(\mu, \gamma)$ is the Kantorovich-Wasserstein
distance between $\mu $ and $\gamma$ given by
$$
{\rm W}_2(\mu, \gamma) \, = \, \inf \bigg ( \int_{\R^n} \! \int_{\R^n} |x-y|^2 d\pi (x,y) \bigg)^{1/2}
$$
the infimum being taken over all couplings $\pi$
on $\R^n \times \R^n$ with respective marginals $\mu$ and $\gamma$.
The proof of \eqref {eq.bgrs} put forward in \cite {BGRS14}
relies on the dimensional self-improved form of the logarithmic Sobolev inequality \cite {BL06}
(see also \cite {BGL14}) expressing that for any smooth density $f$ with respect to $\gamma$,
\beq \label {eq.bl}
{\rm H} (f) \, \leq \, \frac {1}{2} \int_{\R^n} \Delta f \, d\gamma
    + \frac {n}{2} \log \bigg ( 1 + \frac{{\rm I}(f)}{n} 
         - \frac {1}{n} \int_{\R^n} \Delta f \,  d\gamma \bigg).
\eeq
(It is part of the result that the expression inside the logarithm is positive.)
Hence, after a simple rewriting,
$$
\delta (f) \, \geq \, \frac {n}{2} \, \theta \bigg ( \frac {{\rm I}(f) - \int_{\R^n} \Delta f  d\gamma}{n} \bigg)
$$
where $\theta (r) = r - \log (1+r)$, $r > -1$.
Now, by a double integration by parts with respect to the Gaussian density,
$$
\int_{\R^n} |x|^2 d\mu  \, = \, \int_{\R^n} |x|^2 f \, d\gamma 
  \, = \, n + \int_{\R^n} \Delta f \, d\gamma .
$$
Hence, whenever $\int_{\R^n} |x|^2 d\mu \leq n$, then $ \int_{\R^n} \Delta f  d\gamma \leq 0$
and since $\theta$ is increasing,
$$
\delta (f) \, \geq \, \frac {n}{2} \, \theta \bigg ( \frac {{\rm I}(f)}{n} \bigg).
$$

Next we may use again the logarithmic Sobolev inequality
$ {\rm I}(f)  \geq  2 \, {\rm H}(f)$ together with the Talagrand \cite {T96}
quadratic transportation cost inequality (cf.~e.g.~\cite {BGL14, OV00,V09})
\beq \label {eq.talagrand}
 2 \, {\rm H}(f) \, \geq \, {\rm W}_2(\mu, \gamma)^2.
\eeq
Under the condition $\int_{\R^n} |x|^2 d\mu \leq n$,
${\rm W}_2(\mu, \gamma)^2 \leq 4n$, and since $\theta (r) \geq \frac {r^2}{50}$
(for example) on the interval $[0,4]$, the lower bound \eqref {eq.bgrs} follows.
In Section~\ref {sec.6}, we will provide an independent proof of \eqref {eq.bgrs} based
on the information theoretical tools developed in this work.

In another direction, the dimension free lower bound
$$
 \delta (f) \, \geq \, c(\lambda) \, {\rm W}_2(\mu, \gamma)^2
$$
has been established in the recent \cite {FIL14} but under the further assumption that
$\mu $ centered satisfies a Poincar\'e inequality with constant $\lambda >0$.

One drawback of \eqref {eq.bgrs} is of course, besides the dimensional condition
$\int_{\R^n} |x|^2 d\mu \leq n$, that the lower bound depends on $n$ and vanishes
as $n \to \infty$. It is mentioned in
\cite {FIL14} that one cannot expect a dimension free lower bound
only in terms of the Kantorovich-Wasserstein metric ${\rm W}_2$. In addition,
for the extremal $e_b$ of \eqref {eq.extremal}, $\int_{\R^n} |x|^2 e_b d\gamma = n + |b|^2$,
so that the condition $\int_{\R^n} |x|^2 d\mu \leq n$ rules out all extremizers but the
centered one (that is $\gamma$ itself). It is therefore
of interest to look for a measure to the extremizers which may produce stability estimates
independent of the dimension, moreover suitably identifying the extremal densities.

The papers \cite {BGRS14,FIL14}, as well as \cite {CE15}, contain further stability results
involving related transport distances between modifications of $\mu $ and $\gamma$,
however still dimensional. The note \cite {FPR15} presents a lower bound on the deficit based
on a distance (modulo translation) in dimension $n$ starting with a distance
in dimension one first introduced in \cite {BF14}.

\medskip

The aim of this work is to suggest a lower bound on the deficit
$\delta (f)$ in the logarithmic Sobolev inequality in terms of the
Stein characterization of the standard normal distribution $\gamma$.
Before addressing the conclusion, let us first emphasize that, in order to
make sense of $\delta (f)$, it is legitimate to assume that ${\rm H}(f) < \infty$.
Since the deficit should be small, it will also hold that
${\rm I}(f) < \infty$. In particular, this condition entails the fact that the density $f$
has some smoothness, and regularity will be implicitly assumed for the various
expressions to be well-defined. In addition, the
finiteness of ${\rm H}(f)$ ensures by the entropic inequality (see e.g.~\cite[Section 5.1.1]{BGL14})
that
$$
  \int_{\R^n} |x|^2  d\mu \, = \,
\int_{\R^n} |x| ^2f \, d\gamma \, \leq \, 4 \, {\rm H} (f ) 
     + 4 \log \int_{\R^n} \eu^{|x|^2/4} d\gamma < \infty.
$$
Throughout this study of the deficit $\delta (f)$ of the density $f$, it will
therefore be assumed that $\int_{\R^n}|x|^2 d\mu < \infty$
(although this condition is not everywhere strictly necessary). In particular, we may consider
the covariance matrix $\Gamma = {(\Gamma_{ij})}_{1 \leq i, j \leq n}$ of $\mu$ given for
all $i,j = 1, \ldots , n$ by
$$
\Gamma_{ij} \, = \, \int_{\R^n} x_ix_j \, d\mu - 
     \int_{\R^n} x_i \, d\mu  \int_{\R^n} x_j \, d\mu .
$$

\bigskip

The investigation will therefore involve the Stein characterization of the normal distribution,
and more generally ideas related to Stein's method (cf.~\cite {CGS11,NP12,S86}).
Recall indeed the basic integration by parts formula
$$
\int _{\R^n} x \varphi \, d\gamma \, = \, \int_{\R^n} \nabla \varphi \, d\gamma
$$
for any smooth $\varphi : \R^n \to \R$. This equation is characteristic of 
the Gaussian distribution $\gamma$
in the sense that if $\mu $ is a probability measure on $\R^n$
such that for any smooth $\varphi : \R^n \to \R$,
\beq \label {eq.stein}
\int _{\R^n} x \varphi \, d\mu \, = \, \int_{\R^n} \nabla \varphi \, d\mu
\eeq
(as vectors in $\R^n$), then it is necessarily equal to $\gamma$. Indeed, apply for example
\eqref {eq.stein} to $\varphi (x) = \eu^{i \lambda\cdot x}$, $\lambda \in \R^n$,
to get that the Fourier transform $F(\lambda) = \int_{\R^n} \eu^{i \lambda \cdot x} d\mu(x)$,
$\lambda \in \R^n$, of $\mu$ satisfies
$\nabla F = - \lambda F $, hence $F(\lambda) = \eu^{-|\lambda|^2/2}$ and $\mu = \gamma$. 

According to this description, let, for a given probability $\mu$ on $\R^n$ (with finite mean),
\beq \label {eq.distance}
{\cal D} (\mu, \gamma) 
   \, = \, \sup_{\varphi \in {\cal B}} \bigg | \int _{\R^n} \big [ x \varphi - \nabla \varphi \big ] d\mu \bigg|
\eeq
where the supremum runs over the class ${\cal B}$ of smooth functions $\varphi $ on $\R^n$ with
$$
{\| \varphi \|}_\infty \, \leq \, 1, \quad {\| \nabla \varphi \|}_\infty  \, \leq \, 1, \quad 
{\big \| {\rm Hess} ( \varphi ) \big \|}_\infty  \, \leq \, 1.
$$
A more precise class naturally appearing as a family of resolvents (for the Ornstein-Uhlenbeck
semigroup) will be analyzed in Section~\ref {sec.2}, but for the exposition at this
stage, and the comparison with more classical distances, we use the class ${\cal B}$
to state the main results.

\medskip

The main result of this work is a stability estimate
in the logarithmic Sobolev inequality by means of the Stein functional
\eqref {eq.distance}. If $f$ is a probability density
with respect to $\gamma$ with mean $b$, define the shifted probability density
\beq \label {eq.fb}
f_b(x) \, = \,  f(x + b) \, \eu ^{-(b\cdot x + |b|^2/2)} , \quad x \in \R^n,
\eeq
which has mean zero with respect to $\gamma$. (In other words,
if $X$ is a random vector with distribution $f d\gamma$ and
mean $b$, $X-b$ has distribution $f_bd\gamma$ and mean zero.) Then, whenever
$d\mu_b = f_b d\gamma$ is close to $\gamma$, that is $f_b$ is close to the
constant $1$ function, $f(x+b)$ is close to $\eu ^{b\cdot x + |b|^2/2}$,
hence after translation $f$ is close to the extremal $e_b$ of \eqref {eq.extremal}.
Note furthermore that ${\rm I}(f_b) = {\rm I}(f) - |b|^2 \leq {\rm I}(f)$.

\begin {theorem} \label {thm.main1}
Let $f$ be a probability density on $\R^n$ and let $d\mu = f d\gamma$.
Assume that $\mu$ has barycenter $b$ and covariance matrix
$\Gamma \leq {\rm Id} $ (in the sense of symmetric matrices). Then, 
$$
\delta (f) \, \geq \,   \frac {1}{64 (1 + {\rm I}(f_b))^2} \, {\cal D} (\mu_b, \gamma)^4 .
$$
\end {theorem}

As mentioned above, this result will actually be proved for a metric ${\cal D}$ associated to 
a natural class of resolvents as Theorem~\ref {thm.main1bis} below (with in particular
a numerical constant in the lower bound, independent of the Fisher information).

The covariance hypothesis $\Gamma \leq {\rm Id} $ is of course not very natural, although 
the aforementioned investigations implicitely encountered
the same difficulty, and for example \eqref {eq.bgrs} assumes that ${\int_{\R^n} |x|^2 d\mu \leq n}$.
However, with respect to \eqref {eq.bgrs}, the lower bound
in Theorem~\ref {thm.main1} does not involve specifically the dimension $n$
(is actually numerical in the more precise Theorem~\ref {thm.main1bis}),
and moreover identifies the extremal with mean $b$.
Actually, it would already be
of interest to understand how the deficit $\delta$ could control the proximity of the covariance
matrix to the identity.

In Section~\ref {sec.3}, we provide a variation on Theorem~\ref {thm.main1} that
somehow takes into account this deficiency. In particular, it is shown there that
\beq \label {eq.particular}
2 \delta (f)  + \| \Gamma - {\rm Id} \|^2
     \, \geq \,   \frac {1}{64 (1 + {\rm I}(f_b))^2} \, {\cal D} (\mu_b, \gamma)^4
\eeq
where, for an $n \times n$ matrix $A$,
$$
\| A \| \, = \, \sup_{|\alpha| = 1} A \alpha \cdot \alpha
$$
(despite the notation, observe that $\|\cdot\|$ is not a norm). A more precise version (Theorem~\ref {thm.maincovariance}) allows
for a deficit for arbitrary sizes of $\| \Gamma - {\rm Id} \|$.

In another direction, the next result provides a kind of compactness argument
to bound from below the deficit by an unknown constant depending on $f$.

\begin {theorem} \label {thm.main2}
Let $f$ be a probability density on $\R^n$ and let $d\mu = f d\gamma$.
Assume that $\mu$ has barycenter $b$ and that ${\rm I}(f) < \infty$. 
Then
$$
 \delta (f) \, \geq \,   c (f) \, {\cal D} (\mu_b, \gamma) ^4 
$$
where $c(f) >0$ is a constant depending on $f$ only via the uniform integrability
of the family of measures $(\alpha \cdot x)^2 d\mu$, where $\alpha$ runs over
the unit sphere of $\R^n$.
\end {theorem}

\medskip

The distance, or rather measure of proximity in the sense of Stein, ${\cal D}(\mu, \gamma)$, which
is at the core of the present work, may be recast in terms of the Stein kernel
associated with a given distribution, and compared to its discrepancy
as emphasized in \cite {LNP15}. For a centered probability measure $\mu$,
let $\tau_\mu$ be a Stein kernel (matrix) of $\mu$
in the sense that for any smooth $\varphi : \R^n \to \R$,
$$
\int_{\R^n} x \varphi \, d\mu \, = \, \int_{\R^n} \tau_\mu \nabla \varphi \, d\mu
$$
(as vectors in $\R^n$). Then
$$
{\cal D} (\mu, \gamma) 
   \, = \, \sup_{\varphi \in {\cal B}}
    \bigg | \int _{\R^n} \big [ (\tau_\mu - {\rm Id}) \nabla \varphi \big ] d\mu \bigg | .
$$
Recalling the Stein discrepancy between $\mu $ and $\gamma$,
$$
{\rm S}\big ( \mu \, | \, \gamma) 
   \, = \, \bigg ( \int_{\R^n} | \tau_\mu - {\rm Id} |^2 d\mu \bigg)^{1/2}
$$
(where we recall that $|\cdot|$ stands for the Hilbert-Schmidt norm when applied to matrices),
it holds that
\beq \label {eq.discrepancy}
{\cal D} (\mu, \gamma)  \, \leq \,   {\rm S}\big ( \mu \, | \, \gamma) .
\eeq

For the matter of comparison, note from \cite {LNP15} that
$ {\rm W}_2(\mu, \gamma) \leq {\rm S}(\mu, \gamma)$.

As an additional link between the deficit and the Stein characterization, 
the recent \cite {LNP15} points out an improved form of the logarithmic
Sobolev inequality involving the Stein discrepancy $ {\rm S} = {\rm S}\big ( \mu \, | \, \gamma)$
as
$$
{\rm H}(f) \, \leq \, \frac {{\rm S}^2}{2} \, \log \bigg ( 1 + \frac {{\rm I}(f)}{{\rm S}^2} \bigg) .
$$
In terms of the deficit $\delta = \delta (f) = \frac {1}{2} \, {\rm I}(f) - {\rm H}(f)$,
$$
{\rm H}(f)  \, \leq \,  \frac {{\rm S}^2}{2} 
    \log \bigg (1 + \frac { 2 \, {\rm H}(f) + 2 \delta  }{{\rm S}^2} \bigg )
$$
so that if $r = \frac {2 \, {\rm H}(f)}{{\rm S}^2}$, then
$ r \leq \log \big ( 1 + r + \frac {2 \delta}{{\rm S}^2} \big)$, that is
$$
\frac {2 \delta}{{\rm S}^2} \, \geq \, e^r - 1 - r \, \geq \, \frac {r^2}{2} 
  \, = \, \frac {2\, {\rm H}(f)^2}{{\rm S}^4} \, . 
$$
Therefore
$$
\delta (f) \, \geq \, \frac {{\rm H}(f)^2}{{\rm S}^2} \, .
$$

Together with the transportation cost inequality \eqref {eq.talagrand}, we may
therefore state the following corollary,
close in spirit to \eqref {eq.bgrs}.

\begin{proposition}\label{p:stein}
Let $d\mu = f d\gamma $ centered on $\R^n$
with Stein kernel $\tau_\mu$ and associated discrepancy ${\rm S}\big ( \mu \, | \, \gamma)$. Then
\beq \label{eq.ts}
\delta (f) \, \geq \, \frac {{\rm W}_2(\mu, \gamma)^4}{4\, {\rm S}(\mu \, | \, \gamma)^2} \, .
\eeq
\end{proposition}

\begin{remark} In \cite[Theorem 3.2]{LNP15}, it is proved that
$$
{\rm W}_2(\mu, \gamma) \, \leq \,  {\rm S} \big (\mu \, | \, \gamma \big)
    \arccos \Big ( e^{-\frac{{\rm H}(f)}{{\rm S}^2(\mu \, | \, \gamma)}}\Big).
$$
Such a relation allows one to infer that 
$$
{\rm H}(f) \, \geq \,  {\rm S}^2 \, (\mu \, | \, \gamma \, )
\log\bigg (\frac{1}{\cos ({\rm W}_2(\mu, \gamma) \, {\rm S}(\mu \, | \, \gamma)^{-1} )} \bigg),
$$
so that the estimate \eqref{eq.ts} can be slightly improved as 
\beq \label {eq.tal}
\delta(f) \, \geq \,  {\rm S} \big (\mu \, | \, \gamma \big)^2\,
      \log\bigg ( \frac{1}{\cos({\rm W}_2(\mu, \gamma) \, {\rm S}(\mu \, | \, \gamma)^{-1})} \bigg)^2.
\eeq
Notice that, in view of ${\rm W}_2(\mu, \gamma) \leq  {\rm S}(\mu \, | \, \gamma)$,
one has that
$$
\cos(1) \, \leq \, \cos\big ({\rm W}_2(\mu, \gamma) \, 
     {\rm S}\big(\mu \, | \, \gamma \big )^{-1}\big ) \, \leq \,  1.
$$
\end{remark}

\bigskip

The paper is organized as follows. In the next Section~\ref {sec.2}, we describe
properties of the functional ${\cal D}(\mu, \gamma)$, and actually present
an improved form using resolvents of the Ornstein-Uhlenbeck semigroup.
Section~\ref {sec.3} provides the crucial information theoretic tools to analyze
the deficit in terms of ${\cal D}(\mu, \gamma)$, and on which the proof 
of Theorem~\ref {thm.main1} relies. Theorems~\ref {thm.main1} and
\ref {thm.main2} are then established in Sections~\ref {sec.4} and \ref {sec.5} respectively.
The final section is devoted to an alternate proof of the lower bound \eqref {eq.bgrs}
based on the tools of Section~\ref {sec.3}.

\section {Properties of ${\cal D}(\mu, \gamma)$} \label {sec.2}

The Stein functional naturally arising in the proof of
Theorem~\ref {thm.main1} and \ref {thm.main2} will actually be given by
$$
{\widetilde {\cal D}} (\mu, \gamma) 
   \, = \, \sup_{\varphi \in {\cal R}} \bigg | \int _{\R^n} \big [ x \varphi - \nabla \varphi \big ] d\mu \bigg|
$$
where ${\cal R} $ is the class of the resolvents (for the Ornstein-Uhlenbeck
semigroup ${(P_t)}_{t \geq 0}$)
$$
\varphi \, = \, R \psi  \, = \, 4 \int_0^\infty \eu^{-4t} P_t \psi \, dt
$$
with $\psi : \R^n \to \R $ (smooth, for example $C^1$) such that
$\int_{\R^n} \psi^2 P_t f d\gamma \leq 1$ for every $t \geq 0$. The value $4$ has no
particular meaning.

It is in particular the purpose of this section to compare ${\widetilde {\cal D}}$ and
${\cal D}$. Before, we collect some general informations on
${\widetilde {\cal D}} (\mu, \gamma) $. Recall that we assume throughout the investigation
that the density $f$ is smooth and that
$  \int_{\R^n} |x|^2 d\mu = \int_{\R^n} |x|^2 f d\gamma < \infty$.
Recall also that the Ornstein-Uhlenbeck semigroup ${(P_t)}_{t \geq 0}$ described
by the integral representation \eqref {eq.ou} is invariant and symmetric with respect to $\gamma$.
Its infinitesimal generator ${\rm L} = \Delta - x \cdot \nabla$ satisfies the integration
by parts formula, for smooth functions $g, h : \R^n \to \R$,
$$
\int_{\R^n} g \, {\rm L} h \, d\gamma \, = \, - \int_{\R^n} \nabla g \cdot \nabla h \, d\gamma.
$$

It should be noted first that the integrals
$ \int _{\R^n}  [ x \varphi - \nabla \varphi  ] d\mu$
in the definition of ${\widetilde {\cal D}} (\mu, \gamma) $ are well-defined. If
$\varphi = R\psi $ with
$\int_{\R^n} \psi^2 P_t f d\gamma  = \int_{\R^n} P_t (\psi^2) d\mu \leq 1$ for every $t \geq 0$,
\beqs \begin {split}
\int _{\R^n} | x \varphi | d\mu
   & \, \leq \,  4 \int_0^\infty \eu^{-4t} \int_{\R^n} | x P_t \psi | d \mu \, dt \\
    & \, \leq \,  4  \bigg ( \int_{\R^n} |x|^2 d\mu \bigg)^{1/2}
   \int_0^\infty \eu^{-4t} \bigg (\int_{\R^n} (P_t \psi )^2 d \mu \bigg)^{1/2} \, dt \\
      & \, \leq \,    \bigg ( \int_{\R^n} |x|^2 d\mu \bigg)^{1/2} < \infty.
\end {split} \eeqs
Next, after integration by parts in the integral representation \eqref {eq.ou} of $P_t$,
for every $x \in \R^n$,
\beqs \begin {split}
| \nabla P_t \psi |^2 (x)
    & \, = \, \sum_{i=1}^n (\partial_i P_t \psi )^2 (x) \\
       & \, = \,  \eu^{-2t} \sum_{i=1}^n
      \bigg ( \int_{\R^n} \partial_i \psi \big ( \eu^{-t} x + \sqrt {1 - \eu^{-2t}} \, y \big ) 
                     d\gamma (y) \bigg)^2  \\
   & \, = \,  \frac {\eu^{-2t}}{1- \eu^{-2t}} \sum_{i=1}^n
      \bigg ( \int_{\R^n} y_i \, \psi \big ( \eu^{-t} x + \sqrt {1 - \eu^{-2t}} \, y \big ) 
              d\gamma (y) \bigg)^2  \\
      &\, \leq \, \frac {\eu^{-2t}}{1- \eu^{-2t}} \, P_t (\psi^2) (x)  \\
\end {split} \eeqs
so that
$$
\int_{\R^n} | \nabla \varphi | d\mu
  \, \leq  \, 4 \int_0^\infty \eu^{-4t} \int_{\R^n} |\nabla P_t \psi | d\mu \, dt
    \, \leq \, \, 4 \int_0^\infty  \frac {\eu^{-5t}}{\sqrt {1 - \eu^{-2t}}} \, dt  \, \leq \, 4.
$$

The family $\mathcal {R}$ is a determining class in the sense that
whenever $ {\widetilde {\cal D}} (\mu, \gamma)  = 0$, then $\mu = \gamma$. To check this claim, choose
$\psi (x) = \eu^{i \lambda \cdot x}$, $\lambda \in \R^n$ (rather their real and imaginary parts), so that
$$
\varphi (x) \, = \, R\psi  (x) \, = \, 4 \, \eu^{- |\lambda|^2/2}
    \int_0^1 u^3 \, \eu^{i u\lambda \cdot x + |\lambda|^2u^2/2} du, \quad x \in \R^n.
$$
With $F$ the Fourier transform of $\mu$,
$$
\int_{\R^n} \big [x \varphi - \nabla \varphi \big ] d\mu 
  \, = \, -4  i \, \eu^{- |\lambda|^2/2} 
      \int_0^1 u^3 \, \eu^{  |\lambda|^2u^2/2} \big [  \nabla F(u \lambda )
         + u\lambda F(u\lambda) \big ] du.
$$
If the left-hand side of this identity is zero, after the change of $u$ into $\frac {u}{|\lambda |}$,
$$
      \int_0^{\rho} u^3 \, \eu^{ u^2/2} \big [ \nabla F (u \theta )
         + u\lambda F (u\theta) \big ] du \, = \, 0
$$
where $\rho = |\lambda|$ and $\theta = \frac {\lambda}{|\lambda|}$, $\lambda \not=0$.
This relation holding true for any $\rho >0$ and $\theta \in \sph^{n-1}$, it follows that
$ \nabla F (w ) + w F (w) = 0$ for any $w \in \R^n$, and thus $F $ is the Fourier transform of 
the standard normal $\gamma$ on $\R^n$. 

\medskip

The next proposition is the announced comparison between
${\cal D}(\mu, \gamma)$  and ${\widetilde {\cal D}} (\mu, \gamma)$.

\begin {proposition} \label {prop.alternate}
Let $d\mu = f d\gamma$. Then
$$
{\cal D} (\mu, \gamma ) \, \leq \, 2 \big ( 1 + {\rm I}(f) \big)^{1/2} \,
     {\widetilde {\cal D}} (\mu, \gamma) .
$$
\end {proposition}

\begin {proof} Given $\varphi \in \mathcal{B}$, it is straightforward to check
that $\varphi = R\psi$ where 
$$
\psi  \, = \,  -\frac14 \, ({\rm L} - 4\,{\rm Id}) \varphi.
$$
Hence the condition $\int_{\R^n} \psi^2 P_t f d\gamma  \leq 1$ for every $t \geq 0$
in the definition of $  {\widetilde {\cal D}} (\mu, \gamma)$  turns into
$$
\int_{\R^n} ({\rm L}\varphi - 4 \varphi)^2 P_t f \, d\gamma \, \leq \, 16.
$$
Developing the square, we examine successively the three terms under the boundedness
assumptions on $\varphi$ and its derivatives.

The integral
$ \int_{\R^n} \varphi^2 P_t f d\gamma$ may be controlled by a uniform bound on $\varphi$.
To handle $$ \int_{\R^n}  {\rm L}\varphi \, \varphi P_t f d\gamma,$$ write by integration by
parts that
\beqs \begin {split}
\int_{\R^n}  {\rm L}\varphi \, \varphi P_t f d\gamma
  &  \, = \, - \int_{\R^n}  \nabla \varphi \cdot \nabla( \varphi P_t f) d\gamma \\
  & \, = \, - \int_{\R^n}  |\nabla \varphi|^2 P_t f \, d\gamma
      - \int_{\R^n}   \varphi \nabla \varphi \cdot \nabla P_t f \, d\gamma  .\\
\end {split} \eeqs
Now, if ${\| \varphi \|}_\infty \leq 1$ and ${\| \nabla \varphi \|}_\infty \leq 1$,
by the Cauchy-Schwarz inequality and the exponential decay recalled in \eqref {eq.fisherdecay} below,
$$
\bigg | \int_{\R^n}   \varphi \nabla \varphi \cdot \nabla P_t f \, d\gamma \bigg |
   \, \leq \, \int_{\R^n}   | \nabla P_t f | d\gamma 
   \, \leq \, \sqrt { {\rm I} (P_t f)}  \, \leq \, \eu^{-t}  \sqrt {{\rm I} (f)}.
$$
As a consequence, for every $t \geq 0$,
$$
\bigg | \int_{\R^n}  {\rm L}\varphi \, \varphi P_t f \, d\gamma \bigg |
   \, \leq \, 1 + \sqrt {{\rm I}(f)} \,  .
$$
Finally, again by integration by parts,
\beqs \begin {split}
    \int_{\R^n}  ({\rm L}\varphi )^2 P_t f \, d\gamma
         & \, = \,  - \int_{\R^n} \nabla \varphi \cdot \nabla ( {\rm L}\varphi  \, P_t f) d\gamma \\
        & \, = \,  - \int_{\R^n}  {\rm L}\varphi \, \nabla \varphi \cdot \nabla  P_t f \, d\gamma 
                 - \int_{\R^n}  \nabla \varphi \cdot \nabla {\rm L}\varphi \,   P_t f \, d\gamma  \\
\end {split} \eeqs
Under ${\| \nabla \varphi \|}_\infty \leq 1$,
$$
- \int_{\R^n}  {\rm L}\varphi \, \nabla \varphi \cdot \nabla  P_t f \, d\gamma 
 \, \leq \,  \int_{\R^n}  |{\rm L}\varphi | | \nabla  P_t f | d\gamma
   \, \leq \, \frac {1}{2} \int_{\R^n}  ({\rm L}\varphi )^2  P_t f  \, d\gamma
        + \frac12 \, {\rm I}(P_t f) .
$$
Therefore
$$
 \int_{\R^n}  ({\rm L}\varphi )^2 P_t f d\gamma
    \, \leq \,  - 2 \int_{\R^n}  \nabla \varphi \cdot \nabla {\rm L}\varphi \,   P_t f \, d\gamma
     +  {\rm I}(P_t f) .
$$
Now
$$
- 2 \, \nabla \varphi \cdot \nabla {\rm L}\varphi 
  \, = \, 2 \, \Gamma_2(\varphi) - {\rm L} \big ( | \nabla \varphi |^2 \big)
$$
where
$\Gamma_2(\varphi) =  | {\rm Hess}(\varphi) |^2 + | \nabla \varphi |^2    \leq  2 $.
Once more by integration by parts,
$$
\int_{\R^n}  {\rm L} \big ( | \nabla \varphi |^2 \big)  P_t f d\gamma
  \, = \, -  \int_{\R^n}  \nabla \big ( | \nabla \varphi |^2 \big) \cdot \nabla P_t f d\gamma
$$
and 
$$
\big | \nabla \big ( | \nabla \varphi |^2 \big) \big | 
    \, \leq \, 2 | \nabla \varphi | \,  \big | {\rm Hess}(\varphi) \big | \, \leq \, 2. $$
Altogether, if follows that
$$
 \int_{\R^n}  ({\rm L}\varphi )^2 P_t f d\gamma
      \, \leq \, 4 + 2 \sqrt {{\rm I}(f) } +  \, {\rm I}(f).
$$

As a consequence of the preceding three upper bounds, and using that
$\sqrt{I(f)}\leq\frac12 (1+I(f))$, for any $t \geq 0$,
$$
\int_{\R^n} ({\rm L}\varphi - 4 \varphi)^2 P_t f d\gamma 
     \, \leq \, 33 + 6\, {\rm I}(f) \, \leq \, 64\big(1+I(f)\big).
$$
As a result, $\varphi\in\mathcal{B}$ implies that $\varphi\in\,2\sqrt{1+I(f)}\,\mathcal{R}$, and
the proof of the proposition is completed by homogeneity.
\end {proof}

\section {Information theoretic representation} \label {sec.3}

This section develops the information tools towards a suitable expression
for the deficit in the semigroup formulation \eqref {eq.representation}.

Recall the Ornstein-Uhlenbeck semigroup ${(P_t)}_{t \geq 0}$ from
\eqref {eq.ou}. Note that, as vector valued functions,
provided $f : \R^n \to \R$ is smooth, $\nabla P_t f = \eu^{-t} P_t (\nabla f)$.
It is immediate on the integral representation \eqref {eq.ou} to observe that, by integration by parts,
for every $t \geq 0$, as vectors in $\R^n$,
\beq \label {eq.pt}
P_t (xf) \,= \, \eu^{-t} x P_t f + (1- \eu^{-2t}) P_t (\nabla f) 
     \, = \  \eu^{-t} x P_t f + 2 \, {\rm sh} (t) \nabla P_t f . 
\eeq
From \eqref {eq.pt} is deduced an alternate description of the Fisher information 
$$ 
 {\rm I} (P_t f) \, = \, \int_{\R^n} \frac{|\nabla P_t f|^2}{P_t f} \, d\gamma 
$$
along the semigroup as
$$ 
4 \,{\rm sh}^2(t) \, {\rm I} (P_t f) 
   \, = \,  \int_{\R^n} \frac{ |P_t (xf) - \eu^{-t} x P_t f |^2}{P_t f} \, d\gamma .
$$

Given a probability density $f$ with respect to $\gamma $,
let $X$ be a random vector with distribution $d \mu = fd\gamma $.
Let furthermore $N$ be independent
with law $\gamma$, and set, for every $t \geq 0$,
$$
X_t \, = \,  \eu^{-t} X + {\sqrt {1 - \eu^{-2t}}} \, N .
$$
Note that $X_t$ has distribution $P_t f d\gamma$ since
for any bounded measurable $\varphi : \R^n \to \R$,
\beqs \begin {split}
 \E \big ( \varphi (X_t) \big ) 
   & \, = \,  \int_{\R^n} \! \int_{\R^n} \varphi \big ( \eu^{-t} x + \sqrt {1 - \eu^{-2t}} \, y \big)
           f(x) d\gamma (x) d\gamma (y) \\
    &  \, = \,  \int_{\R^n} P_t \varphi  \, f \, d\gamma  \, = \, \int_{\R^n}  \varphi  \, P_t f \, d\gamma .  \\
\end {split} \eeqs

The next observation is that if $u = \frac{P_t (xf)}{P_t f} : \R^n \to \R^n$, then
$$ 
 \E \big (X \, | \, X_t \big )  \, = \,  u(X_t).
$$
Indeed, for any bounded measurable $\varphi : \R^n \to \R$,
$$ 
\E \Big ( \varphi (X_t) \, \E \big (X \, | \, X_t \big ) \Big ) 
    \, = \,  \E \big ( X  \varphi ( X_t) \big ) 
   \, = \, \int_{\R^n} x P_t \varphi  f \, d\gamma  
    \, = \, \int_{\R^n} \varphi  P_t (x f)  d\gamma
$$
while
$$ 
\E \big ( \varphi (X_t) u(X_t) \big ) 
   \, = \, \int_{\R^n} P_t (\varphi u)  f \, d\gamma 
    \, = \, \int_{\R^n} \varphi \, u P_t f  \, d\gamma  
 $$
from which the announced claim follows.

As a consequence, for every $t \geq 0$
\beq  \label {eq.fishermmse}
4 \,{\rm sh}^2(t) \, {\rm I} (P_t f) 
    \, = \,  \int_{\R^n} P_t f \, \bigg | \frac {P_t (xf)}{P_t f} - \eu^{-t} x \bigg |^2 \, d\gamma 
    \, = \, \E \Big (  \big | \E \big ( X \, | \, X_t \big) - \eu^{-t} X_t \big |^2 \Big).
\eeq

In this Ornstein-Uhlenbeck context, the representation \eqref {eq.fishermmse} of the Fisher
information is the analogue of the Minimal Mean-Square Error
(MMSE) emphasized in \cite {GSV05,GWSV11,L15,NPS14}.
The proximity with the linear estimator $ X_t^\ell = \eu^{-t} X_t $ will turn out essential
in the further developments. In particular, \eqref {eq.fishermmse} rewrites as
$$
4 \,{\rm sh}^2(t) \, {\rm I} (P_t f) 
    \, = \, \E \Big (  \big | \E \big ( X - X_t^\ell \, | \, X_t \big)  \big |^2 \Big), \quad t \geq 0.
$$
At this stage, it might be of interest to point out that
if $  \E  (X \, | \, X_t ) = X_t^\ell$ (for some $t>0$), then $X$ must be standard normal.
Indeed, under this assumption, for any smooth $\varphi : \R^n \to \R$,
$$
\E \big ( \varphi (X_t) X \big ) \, = \, \E \big ( \varphi (X_t) X_t^\ell \big ),
$$
that is
$$
(1-\eu^{-2t}) \, \E \big ( X \, \varphi (X_t) \big ) \, = \, 
   \eu^{-t} \, \sqrt {1 - \eu^{-2t}} \; \E \big ( N \, \varphi (X_t) \big ) .
$$
After integration by parts with respect to $N$,
$$
 \, \E \big ( X \, \varphi (X_t) \big ) \, = \, 
            \eu^{-t}  \, \E \big ( \nabla \varphi (X_t) \big ) .
$$
For $\varphi (x) = \eu^{i \lambda \cdot x}$, $\lambda \in \R^n$, this amounts
again to the differential equation $ \nabla F = -  \lambda F $ for the
Fourier transform $F (\lambda) = \E  (  \eu^{i \lambda \cdot X} ) $ of $X$.

\medskip

We next investigate the analogue of \eqref {eq.fishermmse}
for the time derivative of the Fisher information
$ {\rm I} (P_t f) $. Actually, this derivative is at the root of the representation formula
\eqref {eq.representation} that we recall here
\beq \label {eq.representation2}
 {\rm H} (f)  \, = \,  \frac {1}{2} \, {\rm I} (f) - 
      \int_0^\infty \! \int_{\R^n} P_t f  \, \big | {\rm Hess} (\log P_t f) \big |^2 d\gamma \, dt.
\eeq
Indeed, de Bruijn's formula first expresses that
$$
 \frac{d}{dt} \, {\rm H} (P_t f) \, = \, - \int_{\R^n} P_t f | \nabla \log P_t f|^2 d\gamma
    \, = \, - \,  {\rm I} (P_t f)
$$
so that
$$
{\rm H}(f)  \, = \, \int_0^\infty   {\rm I} (P_t f) dt.
$$  
At the second order, following the $\Gamma$-calculus as exposed e.g.
in \cite {B94,BGL14},
\beq \begin {split} \label {eq.diffeq}
 \frac{d}{dt} \, {\rm I} (P_t f)  
       & \,= \, - 2 \int_{\R^n} P_t f \, \Gamma _2(\log P_t f) d\gamma \\
       & \, = \, - 2 \int_{\R^n} P_t f \Big [ \big | {\rm Hess} (\log P_t f) \big |^2
               + \big | \nabla (\log P_t f) \big |^2 \Big ] d\gamma   \\
        & \, = \, - 2 \int_{\R^n} P_t f  \, \big | {\rm Hess} (\log P_t f) \big |^2 d\gamma 
                 - 2 \, {\rm I}(P_t f) .   \\
\end {split} \eeq
Note, as is classical, that this differential equation implies the exponential decay of the
Fisher information
\beq \label {eq.fisherdecay}
{\rm I} (P_t f) \, \leq \, \eu^{-2t} \, {\rm I} ( f) , \quad t \geq 0.
\eeq
By integration by parts, it follows from \eqref {eq.diffeq} that
$$
{\rm H}(f)  \, = \, \int_0^\infty  \! \eu^{-2t} \big ( \eu^{2t} \, {\rm I} (P_t f) \big )dt
    \, = \, \frac {1}{2} \, {\rm I}(f) 
       + \frac {1}{2} \int_0^\infty  \eu^{-2t} \, \frac{d}{dt} \big ( \eu^{2t} \, {\rm I} (P_t f) \big )dt
$$
and hence \eqref {eq.representation2}. 

Accordingly, in the study of the deficit $\delta (f) = \frac {1}{2} \, {\rm I}(f) - {\rm H}(f)$,
we are therefore interested into
$$  
\int_{\R^n} P_t f \,  \big | {\rm Hess} (\log P_t f) \big |^2 d\gamma 
    \, = \,  \int_{\R^n} P_t f \,\bigg | \frac{ {\rm Hess} (P_t f)}{P_t f}
        - \frac{{\nabla P_t f \otimes \nabla P_t f }}{(P_t f)^2 } \bigg|^2 d\gamma  . 
$$
We analyze this expression as the Fisher information in \eqref {eq.fishermmse}.
Taking partial derivative $\partial _j$ in \eqref {eq.pt} first yields that
$$ 
 \partial _j P_t (x_i f) \, = \,  \eu^{-t} \delta_{ij} P_t f + 
   \eu^{-t} x_i \partial _j P_t f + 2\, {\rm sh}(t)  \partial _{ij} P_t f 
$$
for all $i, j = 1, \ldots, n$. After a further use of \eqref {eq.pt},
$$  
2\, {\rm sh}(t) \partial _j P_t (x_i f) 
  \, = \,  (1 - \eu^{-2t}) \delta_{ij} P_t f + \eu^{-t} x_i P_t (x_j f) - \eu^{-2t} x_i x_j P_t f
    + 4\, {\rm sh}^2(t)  \partial _{ij} P_t f .
$$
Applying then \eqref {eq.pt} one more time but to $x_i f$ for every $i$, we finally get that
$$ 
4\, {\rm sh}^2(t)  \partial _{ij} P_t f \, = \,   P_t (x_i x_j f)   - (1 - \eu^{-2t}) \delta_{ij} P_t f
  - \eu^{-t} \big [ x_i P_t (x_j f) + x_j P_t (x_i f) \big] + \eu^{-2t} x_i x_j P_t f .
$$
  
On the other hand, always from \eqref {eq.pt}, for all $i, j = 1, \ldots, n$,
$$ 
4\, {\rm sh}^2(t) \partial _i P_t f \partial _j P_t f  \, = \, P_t (x_i f) P_t (x_j f) 
   - \eu^{-t} \big [ x_i P_t (x_j f) +  x_j P_t (x_i f) \big ] P_t f + \eu^{-2t} x_i x_j (P_t f )^2 .
$$
In compact notation, it follows that
$$
4\, {\rm sh}^2(t) \bigg [ \frac{ {\rm Hess} (P_t f)}{P_t f}
        - \frac {\nabla P_t f \otimes \nabla P_t f }{(P_t f)^2 } \bigg ]
        \, = \, \frac {P_t (x \otimes x \, f)}{P_t f}   -  \frac {P_t(xf)}{P_t f} \otimes \frac {P_t(xf)}{P_t f} 
         - (1 - \eu^{-2t}) \, {\rm Id} .
$$
         
Recall that if $u = \frac {P_t (xf)}{P_t f} : \R^n \to \R^n$, then
$  \E \big (X \, | \, X_t \big )  = u(X_t)$ (as vectors). Exactly in the same way,
if $ v = \frac {P_t (x \otimes x \, f)}{P_t f} $, then
$$ 
\E \big (X \otimes X \, | \, X_t \big )  \,=\,  v(X_t)
$$
as $n \times n $ matrices. Hence (recall $\E (|X|^2) < \infty$), setting
\beq \label {eq.zt}
Z_t \, = \,  {\rm Cov} \big ( X \, | \, X_t \big) \, = \,
    \E \big (X \otimes X \, | \, X_t \big ) - \E \big (X \, | \, X_t \big ) \otimes \E \big (X \, | \, X_t \big ) ,
\eeq
it holds
$$
16 \, {\rm sh}^4(t)  \int_{\R^n} P_t f  \,  \big | {\rm Hess} (\log P_t f) \big |^2 d\gamma 
    \, = \,  \E \Big ( \big | Z_t - (1 - \eu^{-2t}){\rm Id} \big |^2 \Big ).
$$

From \eqref {eq.representation2}, we may therefore emphasize
the following identity which will be the cornerstone for the analysis of the deficit.

\begin {proposition} \label {prop.deficit}
Under the preceding notation,
\beq \label {eq.deficitzt}
\delta (f) \, = \,
    \int_0^\infty  \frac {1}{16 \, {\rm sh}^4(t)} \,
     \E \Big ( \big | Z_t - (1 - \eu^{-2t}){\rm Id} \big |^2 \Big ) dt.
\eeq
\end {proposition}

\section {Proof of Theorem~\ref {thm.main1}} \label {sec.4}

On the basis of Proposition~\ref {prop.deficit}, we address in this section the proof of
Theorem~\ref {thm.main1}. According to Proposition~\ref {prop.alternate},
we actually establish a lower bound on the deficit
in terms of the functional ${\widetilde {\cal D}}(\mu_b, \gamma)$.

\begin {theorem} \label {thm.main1bis}
Let $f$ be a probability density on $\R^n$ and let $d\mu = f d\gamma$.
Assume that $\mu$ has barycenter $b$ and covariance matrix
$\Gamma \leq {\rm Id} $ (in the sense of symmetric matrices). Then, 
$$
\delta (f) \, \geq \,  \frac {1}{4} \,  {\widetilde  {\cal D}} (\mu_b, \gamma)^4 .
$$
\end {theorem}

\begin {proof}
Since $\delta (f) = \delta (f_b)$ where $f_b$ is the shifted density from \eqref {eq.fb},
it is enough to deal with the centered case $b=0 = \E(X)$.

Fix $t \geq 0$. Recall the linear estimator $X_t^\ell = \eu^{-t} X_t$. Observe that
$$
Z_t \, = \,  {\rm Cov} \big ( X \, | \, X_t \big ) \, = \, {\rm Cov} \big ( X - X_t^\ell \, | \, X_t \big) 
$$
so that
\beq \label {eq.covariance}
\E \Big ( \big | Z_t - (1 - \eu^{-2t}){\rm Id} \big |^2 \Big )
  \, = \, \E \Big ( \big | {\rm Cov} \big ( X - X_t^\ell \, | \, X_t \big)
        - (1 - \eu^{-2t}){\rm Id} \big |^2 \Big ).
\eeq
By Jensen's inequality
\beq \label {eq.jensen}
\E \Big ( \big | Z_t - (1 - \eu^{-2t}){\rm Id} \big |^2 \Big )
   \, \geq \, \sum_{i,j=1}^n \Big [ \E (U_i U_j) - (1 - \eu^{-2t})^2 
        ( \Gamma_{ij} - \delta_{ij} ) \Big]^2
\eeq
where $U_i$, $i=1 , \ldots n$, are the coordinates of the vector
$ U = \E ( X - X_t^\ell  \, | \, X_t)$. 

Assume therefore that $\Gamma \leq {\rm Id}$. Hence
$$
\E \Big ( \big | Z_t - (1 - \eu^{-2t}){\rm Id} \big |^2 \Big )
   \, \geq \, \sum_{i,j=1}^n \big [ \E (U_i U_j) \big ] ^2.
$$
In particular, for every unit vector $\alpha = (\alpha_1, \ldots , \alpha_n) $ in $\R^n$,
$$
\sum_{i,j=1}^n \big [ \E (U_i U_j) \big ] ^2 \, \geq \, 
  \bigg [ \E \bigg ( \sum_{i,j=1}^n \alpha_i \alpha_j U_i U_j \bigg) \bigg ] ^2
$$
so that
$$
\E \Big ( \big | Z_t - (1 - \eu^{-2t}){\rm Id} \big |^2 \Big )
   \, \geq \, \Big [ \E \Big ( \big [ \alpha \cdot \E \big ( X - X^\ell _t  \, | \, X_t \big ) \big ]^2 \Big) \Big ]^2.
$$
For $\alpha$ fixed, by duality, for any smooth $ \psi : \R^n \to \R$ such that
$\E (\psi (X_t)^2) \leq 1$, 
$$
\E \Big ( \big [ \alpha \cdot \E \big ( X - X^\ell _t  \, | \, X_t \big ) \big ]^2 \Big) \, \geq \, 
\Big [ \E \Big ( \psi (X_t) \,  \alpha \cdot  \big [  \E \big (X - X^\ell _t  \, | \, X_t \big)  \big ] \Big) \Big ]^2.
$$
Now,
\beqs \begin {split}
\E \Big ( \psi (X_t) \,  \alpha \cdot \, & \big [  \E \big (X - X_t^\ell     \, | \, X_t \big)  \big ]   \Big) \\
 & \, = \, \E \big ( \psi (X_t) \,  \alpha \cdot  [ X - \eu^{-t} \, X_t ] \big) \\ 
 & \, = \, (1 - \eu^{-2t} ) \, \E \big ( \alpha \cdot X  \, \psi (X_t) \big) 
      - \eu^{-t} \,  \sqrt {1 - \eu^{-2t}} \; \E \big ( \alpha \cdot N  \, \psi (X_t) \big) \\
   & \, = \, (1 - \eu^{-2t} )  \,  \alpha \cdot \E  \big ( X \psi (X_t) 
      -  \eu^{-t} \, \nabla \psi (X_t) \big )    \\
\end {split} \eeqs
where integration by parts with respect to $N$ is performed in the last step.
Taking the supremum over all unit vectors $\alpha$, it follows that
for any (smooth) $ \psi : \R^n \to \R$ such that $\E (\psi (X_t)^2) \leq 1$,
$$
\E \Big ( \big | Z_t - (1 - \eu^{-2t}){\rm Id} \big |^2 \Big )
        \, \geq \,  (1 - \eu^{-2t} )^4\,  \Big |  \E \big ( X \psi (X_t) 
      -  \eu^{-t} \, \nabla \psi (X_t) \big)   \Big |^4 .
$$

Switching back to
semigroup notation and recalling that $X_t$ has distribution $P_tf d\gamma$,
$$
\E \Big ( \big | Z_t - (1 - \eu^{-2t}){\rm Id} \big |^2 \Big )
        \, \geq \, (1-\eu^{-2t})^4 \bigg | \int_{\R^n} \big [x P_t \psi - \nabla P_t \psi \big] d\mu \bigg|^4
$$
where we recall that $d \mu = f d\gamma$.
From \eqref {eq.deficitzt}, we therefore obtain that
$$
\delta (f) \, \geq \,  \int_0^\infty \eu^{-4t} \, \bigg | \int_{\R^n} \big  [ x P_t \psi 
            -  \nabla P_t \psi \big ] d\mu \bigg|^4 dt.
$$
By Jensen's inequality in the $t$ variable,
$$
\delta (f) \, \geq \,  \frac {1}{4} \, \bigg | \int_{\R^n} \big [ x \varphi -  \nabla \varphi \big] d\mu \bigg|^4 
$$
where $\varphi = R\psi $. The proof of Theorem~\ref {thm.main1bis} is thus complete.
\end {proof}

\bigskip

To conclude this section, we present a variation on Theorems~\ref {thm.main1}
and \ref {thm.main1bis}
which somehow takes into account the covariance condition.

Consider, for each $0 < \varepsilon \leq 1$, the modified class
${\cal R}_\varepsilon$ consisting of the functions
$$
\varphi \, = \, R_\varepsilon \psi  \, = \,  4 \int_s^\infty \eu^{-4t} P_t \psi \, dt
$$
with $\eu^{-4s} = \varepsilon$ and
$\int_{\R^n} \psi^2 P_t f d\gamma \leq 1$ for every $t \geq 0$.
Note that ${\cal R}_1 = {\cal R}$.
It is easily seen that ${\cal R}_\varepsilon$ is a determining class for any $\varepsilon$.
Define accordingly ${\widetilde {\cal D}}_\varepsilon (\mu, \gamma)$.
The following statement covers in particular \eqref {eq.particular},
and with the flexibility on $\varepsilon >0$ actually allows for a lower bound on the deficit
independently of the size
of $ \| \Gamma - {\rm Id} \| = \sup_{|\alpha |=1} (\Gamma - {\rm Id}) \alpha \cdot \alpha$.

\begin {theorem} \label {thm.maincovariance}
Let $f$ be a probability density on $\R^n$ and let $d\mu = f d\gamma$.
Assume that $\mu$ has barycenter $b$ and covariance matrix $\Gamma$. Then, 
for every $0 < \varepsilon \leq 1$,
\beq \label {eq.lowerdeltaepsilon}
 2 \delta (f) + \varepsilon \, \| \Gamma - {\rm Id} \|^2  
    \, \geq \,   \frac {1}{4 \,\varepsilon^3} \; {\widetilde {\cal D}}_\varepsilon (\mu_b, \gamma) ^4 .
\eeq
\end {theorem}

Whenever $\delta (f) >0$, a sensible choice for $\varepsilon $ could be
$$
\varepsilon \, = \, \varepsilon (\delta) 
       \, = \, \min \bigg ( 1 , \frac {\delta (f)}{ \| \Gamma - {\rm Id} \|^2 } \bigg)
$$
yielding
\beq \label {eq.selfdeficit}
3\delta (f) 
    \, \geq \,   \frac {1}{4 \,\varepsilon(\delta)^3} \; 
       {\widetilde {\cal D}}_{\varepsilon (\delta)} (\mu_b, \gamma) ^4 
        \, \geq \,   \frac {1}{4} \; 
       {\widetilde {\cal D}}_{\varepsilon (\delta)} (\mu_b, \gamma) ^4 .
\eeq

\begin {proof}
Assume that $b=0$ and start from \eqref {eq.jensen}. For any unit vector $\alpha $ in $\R^n$, 
$$
\E \Big ( \big | Z_t - (1 - \eu^{-2t}){\rm Id} \big |^2 \Big ) \, \geq \,
\Big [ \E\big ( [\alpha \cdot U]^2 \big) - (1 - \eu^{-2t})^2 
    \big [ (\Gamma  - {\rm Id} )\alpha \cdot \alpha \big] \Big] ^2 
$$
itself lower-bounded by
$$
 \frac {1}{2} \, \big [ \E\big ( [\alpha \cdot U]^2 \big) \big]^2
     - 2 (1 - \eu^{-2t})^4  \big [ (\Gamma  - {\rm Id} )\alpha \cdot \alpha \big] ^2.
$$
Arguing as above,
$$
\E \Big ( \big | Z_t - (1 - \eu^{-2t}){\rm Id} \big |^2 \Big ) \, \geq \,
   (1 - \eu^{-2t})^4 \bigg [   \frac {1}{2} \, \bigg | \int_{\R^n} \big [ x P_t \psi
            - \nabla P_t \psi \big ] d\mu \bigg|^4
               - 2 \| \Gamma - {\rm Id} \| ^2\bigg] .
$$

Now, from \eqref {eq.deficitzt} of Proposition~\ref {prop.deficit},
we may bound from below the deficit for every $s>0$ by
$$ 
 \delta (f) \, \geq \,
    \int_s^\infty  \frac {1}{16 \, {\rm sh}^4(t)} \,
     \E \Big ( \big | Z_t - (1 - \eu^{-2t}){\rm Id} \big |^2 \Big ) dt.
$$
Hence, by the preceding,
\beqs \begin {split}
 \delta (f) 
    & \, \geq \, \frac {1}{2} \int_s^\infty \eu^{-4t} \, \bigg | \int_{\R^n} \big  [ x P_t \psi
            -  \nabla P_t \psi \big ] d\mu \bigg|^4 dt
            - 2 \int_s^\infty \eu^{-4t} \,  \| \Gamma - {\rm Id} \| ^2 dt \\
      &\, \geq \, \frac {1}{2} \int_s^\infty \eu^{-4t} \, \bigg | \int_{\R^n} \big  [ x P_t \psi 
            -  \nabla P_t \psi \big ] d\mu \bigg|^4 dt
            - \frac {1}{2} \, \eu^{-4s} \, \| \Gamma - {\rm Id} \|  ^2 .\\
\end {split} \eeqs
By the same Jensen's inequality argument, but now on the interval $(s,\infty)$,
$$
 \delta (f)   \, \geq \,
   \frac {1}{8} \, \eu^{12s}  \bigg | \int_{\R^n} \big  [ x \varphi -  \nabla \varphi \big ] d\mu \bigg|^4 
            - \frac {1}{2} \, \eu^{-4s} \, \| \Gamma - {\rm Id} \| ^2
$$
where now $\varphi = R_\varepsilon \psi$ with 
$\varepsilon = \eu^{-4s}$. Theorem~\ref {thm.maincovariance} then follows.
\end {proof}

\section {Proof of Theorem~\ref {thm.main2}} \label {sec.5}

As for the preceding theorems, we may and do assume that $b = 0$.
From the latter \eqref {eq.lowerdeltaepsilon} (with $\varepsilon = 1$), we get that
\beq \label {eq.lowerbound}
2 \delta (f) \, \geq \,
   \frac {1}{4} \, {\widetilde {\cal D}}(\mu, \gamma)^4 -  \, \big \| \Gamma - {\rm Id} \big \|^2.
\eeq

The challenge now is to control the covariance matrix by the deficit.
Assume therefore that $\| \Gamma - {\rm Id} \| >0$, otherwise
apply Theorem~\ref {thm.main1}. Recalling that
$ Z_t = {\rm Cov}  ( X \, | \, X_t )$, for any $ t \geq 0$ and any unit vector
$\alpha $ in $\R^n$, by Jensen's inequality,
\beqs \begin {split}
\E \Big ( \big | Z_t -   (1 - \eu^{-2t}){\rm Id} \big |^2 \Big ) 
  & \, = \, \E \Big ( \big | {\rm Cov} \big ( X \, | \, X_t \big)
        - (1 - \eu^{-2t}){\rm Id} \big |^2 \Big ) \\
   & \, \geq \,  \Big [   (\Gamma - {\rm Id})\alpha \cdot \alpha + \eu^{-2t}  
              - \E \Big ( \big [ \E \big ( \alpha \cdot X \, | \, X_t \big)  \big ]^2 \Big)   \Big ]^2 . \\
\end {split} \eeqs
By a rough estimate,
$$
\E \Big ( \big | Z_t -   (1 - \eu^{-2t}){\rm Id} \big |^2 \Big ) 
    \, \geq \, \frac {1}{2} \, \big [ (\Gamma - {\rm Id})\alpha \cdot \alpha \big]^2
       - 2 \, \eu^{-4t} - 
       2 \, \Big [ \E \Big ( \big [  \E \big ( \alpha \cdot X \, | \, X_t \big)  \big ] ^2 \Big) \Big ]^2 
$$
and taking the supremum over $\alpha$, for any $t \geq 0$,
\beq \label {eq.rough}
\E \Big ( \big | Z_t -   (1 - \eu^{-2t}){\rm Id} \big |^2 \Big ) 
    \, \geq \, \frac {1}{2} \, \big \| \Gamma - {\rm Id} \big \|^2       - 2 \, \eu^{-4t} - 
       2 \sup_{|\alpha| =1}
        \Big [ \E \Big ( \big [  \E \big ( \alpha \cdot X \, | \, X_t \big)  \big ] ^2 \Big) \Big ]^2 .
\eeq
The following lemma is the compactness argument from which the conclusion will follow.

\begin {lemma}  \label {lem.limit}
Set
$$
\rho (t) \, = \,  \sup_{|\alpha | = 1} 
    \E \Big ( \big [  \E \big ( \alpha \cdot X \, | \, X_t \big)  \big ] ^2 \Big), \quad t \geq 0.
$$
Then, whenever $X$ is centered and ${\rm I}(f) < \infty$, 
$$
\lim_{t \to \infty} \rho (t) = 0.
$$
\end {lemma}

\begin {proof}
Fix first a unit vector $\alpha \in \R^n$, and let $t \geq 0$.
We develop the proof in the semigroup language.
As discussed in Section~\ref {sec.3}, $\E ( \alpha \cdot X \, | \, X_t ) = u_\alpha (X_t)$ where
$ u_\alpha = \frac {P_t (\alpha \cdot x f)}{P_t f} $.
Hence
$$
\E \Big ( \big [  \E \big ( \alpha \cdot X \, | \, X_t \big)  \big ] ^2 \Big)
   \, = \, \int_{\R^n} \frac {  P_t (\alpha \cdot x f)^2}{P_t f} \, d \gamma.
$$
Now,
\beq \begin {split} \label {eq.decomposition}
\int_{\R^n} \frac { P_t (\alpha \cdot x f)^2}{P_t f} \, d \gamma
    & \, \leq \, \int_{\{ |P_t(\alpha \cdot x f)| \leq P_t f \}} 
           \frac {P_t(\alpha \cdot x f)^2}{P_t f}\, d\gamma 
       + \int_{\{ | P_t (\alpha \cdot x f) | \geq P_t f \}} 
            \frac {P_t(\alpha \cdot x f)^2}{P_t f} \, d\gamma\\
     & \, \leq \, \int_{\R^n} \big | P_t (\alpha \cdot x f) \big |  d\gamma 
       +  \int_{A_t} \frac {P_t (\alpha \cdot x f) ^2}{P_t f} \, d\gamma \\
\end {split} \eeq
where $ A_t = \{ |P_t(\alpha \cdot x f)| \geq P_t f \}$. 

Observe that
\beqs \begin {split}
\gamma (A_t) 
   & \, \leq \, \gamma \big ( 2 |P_t(\alpha \cdot x f)|  \geq 1 \big)
                     + \gamma \big ( 2 |P_t f - 1|  \geq 1 \big)  \\
    & \, \leq \, 2 \int_{\R^n} \big | P_t (\alpha \cdot x f) \big |  d\gamma 
         + 2 \int_{\R^n} \big | P_t f - 1 \big |  d\gamma . \\
\end {split} \eeqs 
Since $ P_t (\alpha \cdot x f)$ is centered with respect to $\gamma$,
by the Gaussian ${\rm L}^1$-Poincar\'e inequality,
$$
\int_{\R^n} \big | P_t (\alpha \cdot x f) \big | d \gamma
  \, \leq \, 2 \int_{\R^n} \big | \nabla \big ( P_t (\alpha \cdot x f) \big )\big | d \gamma
  \, \leq \, 2 \, \eu^{-t}  \int_{\R^n} \big | \nabla (\alpha \cdot x f) \big | d \gamma.
$$
Now, since $\alpha $ is a unit vector,
$$
\int_{\R^n} \big | \nabla (\alpha \cdot x f) \big | d \gamma
   \, \leq \, 1 + \int_{\R^n} |\alpha \cdot x| \, |\nabla f|  d \gamma
   \, \leq \, 1 + \bigg (\int_{\R^n} |x|^2 f \, d\gamma \bigg)^{1/2} \, {\rm I}(f)^{1/2}.
$$
Again by the Gaussian ${\rm L}^1$-Poincar\'e and Cauchy-Schwarz inequalities,
$$
\int_{\R^n} \big | P_t f - 1 \big |  d\gamma
   \, \leq \, 2 \int_{\R^n} | \nabla P_t f | d\gamma
       \, \leq \, 2 \,  \sqrt { {\rm I} (P_t f)} \, ,
$$
so that, by \eqref {eq.fisherdecay},
$$
\int_{\R^n} \big | P_t f - 1 \big |  d\gamma  \, \leq \, 2 \, \eu^{-t}  \sqrt { {\rm I} (f)} \, .
$$
These estimates already ensure that, uniformly in $|\alpha | =1$, 
$$
\lim_{t \to 0} \int_{\R^n} \big | P_t (\alpha \cdot x f) \big |  d\gamma  \, = \, 0
$$
and $\lim_{t \to 0} \gamma (A_t) = 0$.

To handle the second term in \eqref {eq.decomposition}, note that
by the Cauchy-Schwarz inequality (for $P_t$),
\beqs \begin {split}
\int_{A_t}\frac {P_t (\alpha \cdot x f) ^2}{P_t f} \, d\gamma
  & \, \leq \, \int_{A_t} P_t \big ( (\alpha \cdot x)^2 f \big) d\gamma \\
   & \, = \, \int_{\R^n} P_t (1_{A_t}) (\alpha \cdot x)^2 f  \, d\gamma \\
   & \, \leq \, \int_{\R^n} P_t (1_{A_t}) |x|^2 f  \, d\gamma. \\
\end {split} \eeqs 
Since $\int_{\R^n} P_t(1_{A_t}) d\gamma = \gamma (A_t)$, the conclusion follows by
dominated convergence. Lemma~\ref {lem.limit} is established.
\end {proof}

On the basis of Lemma~\ref {lem.limit}, we may now conclude the proof of Theorem~\ref {thm.main2}.
Going back to \eqref {eq.rough}, for every $t \geq 0$
$$
\E \Big ( \big | Z_t -   (1 - \eu^{-2t}){\rm Id} \big |^2 \Big ) 
   \, \geq \, \frac {1}{2} \big \|  \Gamma - {\rm Id} \big \| ^2
       - 2 \, \eu^{-4t} - 2 \, \rho (t).
$$
By Lemma~\ref {lem.limit}, choose $t_0$ large enough so that
$$
 2 \, \eu^{-4t} + 2 \, \rho (t)  \, \leq \, \frac {1}{4} \big \|  \Gamma - {\rm Id} \big \| ^2
$$
for every $ t \geq t_0$. Then, for $t \geq t_0$,
$$
\E \Big ( \big | Z_t -   (1 - \eu^{-2t}){\rm Id} \big |^2 \Big ) 
   \, \geq \, \frac {1}{4} \big \|  \Gamma - {\rm Id} \big \| ^2
$$
and
\beq \label {eq.deficitcovariance}
\delta(f) \, \geq \, \int_{t_0}^\infty  \frac {1}{16 \, {\rm sh}^4(t)} 
                              \, \E \Big ( \big | Z_t -   (1 - \eu^{-2t}){\rm Id} \big |^2 \Big ) 
                \, \geq \, \frac {\eu^{-4t_0}}{16} \, \big \|  \Gamma - {\rm Id} \big \| ^2 .
\eeq
The conclusion of the proof of Theorem~\ref {thm.main2} is then a suitable combination of \eqref {eq.deficitcovariance} and \eqref {eq.lowerbound}.

\medskip

It may be observed that the conclusion of Lemma~\ref {lem.limit}
actually amounts to the following result which we present as a statement
of possible independent interest. From the classical exponential decay \eqref {eq.fisherdecay},
$ \eu^{2t} \, {\rm I} (P_t f) \leq  {\rm I} ( f) $ for every $t \geq 0$.
Under the centering $\int_{\R^n} x f d\gamma = 0$, we actually have

\begin {corollary} 
Let $f$ be a smooth probability density with respect to
$\gamma$ such that $\int_{\R^n} x f d\gamma = 0$,
$\int_{\R^n} |x|^2 f d\gamma < \infty$ and ${\rm I}(f) < \infty$. Then
$$
\lim_{t \to \infty} \eu^{2t} \,  {\rm I}(P_t f) \, = \, 0.
$$
\end {corollary}

\section{An alternate proof of the estimate \eqref {eq.bgrs}}  \label {sec.6}

To conclude this work, we present in this section an alternate proof of the lower bound \eqref {eq.bgrs}
\beq \label {eq.bgrs2}
 \delta (f) \, \geq \, \frac {c}{n} \, {\rm W}_2(\mu, \gamma)^4
\eeq
under the condition $\int_{\R^n} |x|^2 d\mu \leq n$ (with the constant $c=\frac14$).

Recall $Z_t$, $t \geq 0$, from \eqref {eq.zt}. By the Cauchy-Schwarz inequality,
for every $t$,
\beq \label {eq.cauchy}
n \,  \E \Big ( \big | Z_t - (1 - \eu^{-2t}){\rm Id} \big |^2 \Big )
    \, \geq \, \Big [ \E \big ( {\rm Tr} \big (Z_t - (1 - \eu^{-2t}){\rm Id} \big)  \big) \Big ]^2.
\eeq
Now, after some straightforward calculations,
$$
\E \big ( {\rm Tr} \big (Z_t - (1 - \eu^{-2t}){\rm Id} \big)  \big) 
     \, = \, \E \big (|X|^2 \big) - \sum_{i=1}^n \E \big ( \E  ( X^i \, | \, X_t )^2 \big) 
          - n (1 - \eu^{-2t})
$$
where $X^i$, $i = 1, \ldots, n$, are the coordinates of the vector $X$.
On the other hand,
$$
\E \Big (  \big | \E \big ( X \, | \, X_t \big) - \eu^{-t} X_t \big |^2 \Big)
  \, = \,  \sum_{i=1}^n \E \big ( \E  ( X^i \, | \, X_t )^2 \big) 
          - \eu^{-2t}  ( 2 - \eu^{-2t} ) \, \E \big (|X|^2 \big) + n \, \eu^{-2t} (1 - \eu^{-2t}).
$$
Therefore, whenever $ \E(|X|^2) = \int_{\R^n} |x|^2 d\mu \leq n $,
$$
- \, \E \big ( {\rm Tr} \big (Z_t - (1 - \eu^{-2t}){\rm Id} \big)  \big) 
   \, \geq \, \E \Big (  \big | \E \big ( X \, | \, X_t \big) - \eu^{-t} X_t \big |^2 \Big).
$$
Hence, by \eqref {eq.fishermmse} and \eqref {eq.cauchy},
$$
n \,  \E \Big ( \big | Z_t - (1 - \eu^{-2t}){\rm Id} \big |^2 \Big )
  \, \geq \, 16 \, {\rm sh}^4(t) \, {\rm I}(P_t f)^2,
$$
and combined with \eqref{eq.deficitzt},
\beq \label{eq.fd}
\delta(f) \, \geq \, \frac {1}{n} \int_0^\infty {\rm I}(P_tf)^2 dt.
\eeq 

Now write $d\mu_t = P_t f\,d\gamma$, $t \geq 0$, and define
$$
w(t) \, = \,  {\rm W}_2( \mu, \mu_t).
$$ 
In particular, $w(\infty) = {\rm W_2}(\mu, \gamma ) $ (while $w(0) = 0$).
Recall from \cite {OV00} that $ w'(t)\leq \sqrt{{\rm I}(P_tf)}$, $t \geq 0$. 
From the transportation cost inequality \eqref {eq.talagrand} applied to
$\mu_t$, $w(t)^2 \leq 2 {\rm H}(P_t f)$, and with the logarithmic Sobolev inequality
$w(t)^2 \leq 2 {\rm H}(P_t f) \leq {\rm I} (P_t f)$. Hence, combining these inequalities,
for every $t \geq 0$,
$$
w(t)^3 w'(t) \, \leq \,  w(t)^3 \sqrt{{\rm I}(P_tf)}  \, \leq \,  {\rm I}(P_tf)^2.
$$
Together with \eqref{eq.fd},
$$
\delta(f) \, \geq \, \frac {1}{n} \int_0^\infty w(t)^3 \,w'(t) dt  \, = \, \frac{1}{4n} \, w(\infty)^4
      \, = \, \frac{1}{4n} \, {\rm W_2}(\mu, \gamma )^4
$$
and therefore the desired conclusion \eqref {eq.bgrs2} with $c = \frac {1}{4}$.

\end {document}